\documentclass{endm}
\usepackage{endmmacro}
\usepackage{graphicx}
\usepackage{pdflscape} 
\usepackage{booktabs}
\usepackage{dcolumn}
\newcolumntype{d}[1]{D{.}{.}{#1}}

\usepackage{pdflscape} 

\usepackage{booktabs}
\usepackage{amssymb}

\usepackage{multirow}

\usepackage{tikz}
\definecolor{mediumspringgreen}{rgb}{0.0, 0.98039215, 0.60392156}

\usetikzlibrary{positioning}

\usepackage[utf8]{inputenc}

\usepackage[english]{babel}
\usepackage{amsfonts}
\usepackage{amsmath}
\usepackage{latexsym}
\usepackage{color}
\usepackage{enumerate}

\usepackage{ifpdf}
\usepackage{listings}
\newcommand\ReviewOnlyText[1]{#1}   

\DeclareMathOperator    \intr                   {int}

\DeclareMathOperator    \relint         {rel\,int}

\DeclareMathOperator    \verts          {vert}

\newcommand{\old}[1]{{}}

\newcommand{\bb}{\mathbb}

\newcommand{\R}{\bb R}

\newcommand{\floor}[1]{\lfloor#1\rfloor}
\newcommand{\ceil}[1]{\lceil #1 \rceil}
\newcommand\st{\mid}

\newcommand{\B}{B}

\newcommand{\setcond}[2]{\left\{ #1 \,\st\, #2 \right\}}
\renewcommand{\P}{\mathcal{P}}

\def\st{\mid}

\newenvironment{psmallmatrixbig}{\bigl(\smallmatrix}{\endsmallmatrix\bigr)}
\newcommand\InlineFrac[2]{#1/#2}  
\newcommand\ColVec[3][\relax]
{
  \ifx#1\relax
  \bgroup\let\frac=\InlineFrac\begin{psmallmatrixbig}#2\vphantom{/}\\#3\vphantom{/}\end{psmallmatrixbig}\egroup
  \else
  \bgroup\let\frac=\InlineFrac\begin{psmallmatrixbig}\ifx#200\else#2/#1\fi\\\ifx#300\else#3/#1\fi\end{psmallmatrixbig}\egroup
  \fi
}

\makeatletter
\renewcommand{\pod}[1]
{\allowbreak\mathchoice{\mkern18mu}{\mkern8mu}{\mkern8mu}{\mkern8mu}(#1)}
\makeatother

\chardef\Myunderscore=`\_
\newcommand\underscore{\Myunderscore\allowbreak}
\let\_=\underscore

\DeclareRobustCommand\sage[1]{\texttt{#1}}

\usepackage[normalem]{ulem}

\newcommand\Figure[2][\relax]{%
  \begin{figure}[h!]
    \includegraphics[width=.8\textwidth]{#2}
    \caption{\ifx#1\relax#2\else#1\fi}
  \end{figure}
}

\graphicspath{{../dff_paper_graphics/}}

\begin{document}

\begin{verbatim}\end{verbatim}\vspace{2.5cm}

\begin{frontmatter}

\title{
Structure and Interpretation \\ 
of Dual-Feasible Functions
}

\author{Matthias K\"oppe\thanksref{ALL}\thanksref{myemail}}
\address{Dept.\@{} of Mathematics\\ University of California, Davis\\ USA}

\author{Jiawei Wang\thanksref{ALL}\thanksref{coemail}}
\address{Dept.\@{} of Mathematics\\ University of California, Davis\\ USA}
\thanks[ALL]{The authors gratefully acknowledge partial support from the National Science
  Foundation through grant DMS-1320051, awarded to M.~K\"oppe.} 
\thanks[myemail]{Email:
  \href{mailto:mkoeppe@math.ucdavis.edu} {\texttt{\normalshape
      mkoeppe@math.ucdavis.edu}}}
\thanks[coemail]{Email:
  \href{mailto:wangjw@math.ucdavis.edu} {\texttt{\normalshape
      jwewang@ucdavis.edu}}}


\begin{abstract}

We study two techniques to obtain new families of classical and general
Dual-Feasible Functions:  A conversion from minimal Gomory--Johnson functions;
and computer-based search using polyhedral computation and an automatic
maximality and extremality test. 

\end{abstract}


\end{frontmatter}

\clearpage
\section{Introduction}

The duality theory of integer linear optimization appears in several concrete
forms.  Inspired by the monograph
\cite{alves-clautiaux-valerio-rietz-2016:dual-feasible-book}, we study
\emph{(classical) Dual-Feasible Functions} (DFFs, cDFFs), which are defined as
functions $\phi\colon D \to D$ such that
$\sum_{i\in I}x_i \le 1 \Rightarrow \sum_{i\in I}\phi(x_i) \le 1$ for any
family $\{x_i\}\subseteq D$ indexed by a finite index set $I$, where
$D=[0,1]$. In \cite{alves-clautiaux-valerio-rietz-2016:dual-feasible-book},
these functions are studied alongside with \emph{general DFFs} (gDFFs), which
satisfy the same property for the extended domain $D=\mathbb{R}$
\ReviewOnlyText{(reviewers are invited to refer to Appendix
  \ref{s:preliminaries} for more details)}.

DFFs appear to first have been studied by Lueker
\cite{Lueker:1983:BPI:1382437.1382833} to provide lower bounds for bin-packing
problems. DFFs can derive feasible solutions to the dual problem of the LP
relaxation efficiently, therefore providing fast lower bounds for the primal
IP problem.  The computation of bounds is also the main angle of exposition in
the monograph \cite{alves-clautiaux-valerio-rietz-2016:dual-feasible-book}.
Vanderbeck \cite{vanderbeck2000exact} studied the use of DFFs in several
combinatorial optimization problems including the cutting stock problem, 
generating valid inequalities for these problems.

The \emph{maximal} (pointwise non-dominated) DFFs are of particular interest since
they provide better lower bounds and stronger valid inequalities
. Maximality is not enough if the strongest bounds and inequalities are
expected. A maximal DFF is said to be \emph{extreme} if it can not be written
as a convex combination of two other maximal DFFs. Therefore, a hierarchy on
the set of valid DFFs, which indicates the strength of the corresponding valid
inequalities and lower bounds, has been defined
\cite{alves-clautiaux-valerio-rietz-2016:dual-feasible-book}.  This
development is parallel to the one in the study of cut-generating functions
\cite{yildiz2016cut}, to which there is a close relation that deserves to be
explored in greater depth.  Indeed, the characterization of minimal
cut-generating functions in the Y{\i}ld{\i}z--Cornu\'ejols model
\cite{yildiz2016cut} can be easily adapted to give a full characterization of
maximal general DFFs, which is missing in
\cite{alves-clautiaux-valerio-rietz-2016:dual-feasible-book}%
\ReviewOnlyText{{} (see Appendix~\ref{s:Yildiz-Conrnuejols})}.

The authors of \cite{alves-clautiaux-valerio-rietz-2016:dual-feasible-book}
study analytical properties of extreme DFFs and use them to prove the
extremality of various classes of functions, most of which are piecewise
linear (possibly discontinuous).

In our paper, we complement this study by transferring recent algorithmic
techniques
\cite{basu-hildebrand-koeppe:equivariant,hong-koeppe-zhou:software-abstract-with-link-to-software}
developed by Basu, Hildebrand, Hong, K\"oppe, and Zhou for 
cut-generating functions in the Gomory--Johnson model \cite{igp_survey} to
DFFs.  In our software, available as the feature 
branch \sage{dual\textunderscore feasible\textunderscore functions} in
\cite{hong-koeppe-zhou:software-abstract-with-link-to-software}, we implement
an automatic maximality and extremality test for classical DFFs.  

In our software, written in SageMath \cite{sage}, a comprehensive Python-based
open source computer algebra system, we also provide an electronic compendium
of the known extreme DFFs from
\cite{alves-clautiaux-valerio-rietz-2016:dual-feasible-book}. We hope that it
facilitates experimentation and further study.

The main objective of our paper is to introduce two methods to build new DFFs
in quantity.  In \autoref{s:abc}, we introduce a conversion from
Gomory--Johnson functions to DFFs, which under some conditions generates
maximal or extreme general and classical DFFs.  The Gomory--Johnson model is
well-studied and the literature provides a large library of known functions.
From our conversion, we obtain 2-slope extreme DFFs and maximal DFFs with
arbitrary number of slopes.  

In \autoref{s:com}, we discuss a computer-based search technique, based on our
automatic maximality and extremality test.  We obtain a library of extreme
DFFs with rational breakpoints in $\frac{1}{q}\mathbb{Z}$ for fixed
$q\in \mathbb{N}$. By using computer-based search we find new extreme DFFs
with intriguing structures.
Our work is a starting point for finding new parametric families of DFFs with
special properties
.

Our methods complement those presented in the monograph
\cite{alves-clautiaux-valerio-rietz-2016:dual-feasible-book}, which have a
more analytical flavor, such as building new DFFs from ``simple'' DFFs by the
operation of composition of
functions
.




\section{Relation to Gomory--Johnson functions} 
\label{s:abc}

In this section, we show
that 
new DFFs, especially extreme ones, can be discovered by converting
Gomory--Johnson functions to DFFs. 
We first introduce the Gomory--Johnson cut-generating functions; details can be found in \cite{igp_survey}. Consider the single-row Gomory--Johnson model, which takes the following form:
\begin{equation}
x+\sum_{r\in\mathbb{R}}r\, y(r)= b, \quad b\notin \mathbb{Z}, b>0
\end{equation}
$$x\in \mathbb{Z},\, y:\mathbb{R} \to \mathbb{Z}_+, \,\text{and $y$ has finite support.}$$
Let $\pi\colon\mathbb{R} \to\mathbb{R}$ be a nonnegative function. Then by
definition $\pi$ is a valid Gomory--Johnson function if
$\sum_{r\in\mathbb{R}}\pi(r)\, y(r)\ge1$ holds for any feasible solution
$(x,y)$. Minimal (valid) functions are characterized by subadditivity and
several other properties.

As maximal DFFs are superadditive, underlying the conversion is that
subtracting subadditive functions from linear functions gives superadditive
functions; but the details are more complicated.
\begin{theorem}
\label{thm:main}
Let $\pi$ be a minimal piecewise linear Gomory--Johnson function corresponding
to a row of the form (1) with the right hand side $b$. Assume $\pi$ is
continuous at $0$ from the right. Then there exists $\delta>0$, such that for
all $0<\lambda<\delta$, the function $\phi_\lambda\colon\mathbb{R} \to\mathbb{R}$, defined
by $\phi_\lambda(x)=\frac{bx-\lambda\pi(bx)}{b-\lambda}$, is a maximal general
DFF and its restriction $\phi_\lambda|_{[0,1]}$ is a maximal classical DFF.
These functions have the following properties.

  \begin{enumerate}[(i)]
  \item[(i)] 
  $\pi$ has $k$ different slopes if and only if $\phi_\lambda$ has $k$ different slopes. If $b>1$, then $\pi$ has $k$ different slopes if and only if $\phi_\lambda|_{[0,1]}$ has $k$ different slopes.
   \item[(ii)] 
     The gDFF $\phi_\lambda$ is extreme if $\pi$ is also continuous with only
     2 slope values where its positive slope $s$ satisfies $sb>1$ and
     $\lambda=\frac{1}{s}$. The cDFF $\phi_\lambda|_{[0,1]}$ is extreme if
     $\pi$ and $\lambda$ satisfy the previous conditions and $b>3$. 
  \end{enumerate}
\end{theorem}
\begin{proof} 
\ReviewOnlyText{See detailed proof in Appendix \ref{s:proof}.} As a minimal
valid Gomory--Johnson function, $\pi$ is $\mathbb{Z}$-periodic, $\pi(0)=0$,
$\pi$ is subadditive and $\pi(x)+\pi(b-x)=1$ for all $x\in \mathbb{R}$ (see
\cite{igp_survey}). It is not hard to check $\phi_\lambda(0)=0$,
$\phi_\lambda$ is superadditive and $\phi_\lambda(x)+\phi_\lambda(1-x)=0$ for
all $x\in \mathbb{R}$. If $\lambda$ is small enough, there exists an
$\epsilon>0$ such that $\phi_\lambda(x)\ge 0$ for all $x\in(0,
\epsilon)$. Therefore, $\phi_\lambda$ is a maximal general DFF and
$\phi_\lambda|_{[0,1]}$ is a maximal classical DFF, using the characterization
of maximality in \cite{alves-clautiaux-valerio-rietz-2016:dual-feasible-book}. 

\smallbreak

\noindent\emph{Part (i)}. 
Suppose $\pi$ has slope $s$ on the interval $(a_i, a_{i+1})$, then by calculation $\phi_\lambda(x)$ has slope $s'=\frac{b(1-\lambda s)}{b-\lambda}$ on the interval $(\frac{a_i}{b},\frac{a_{i+1}}{b})$. From the fact we can conclude $\pi$ has $k$ different slopes if and only if $\phi_\lambda$ has $k$ different slopes.  Since $\pi$ is $\mathbb{Z}$-periodic, $\phi_\lambda$ is quasiperiodic with period $\frac{1}{b}$. If $b>1$, the interval $[0,1]$ contains a whole period, so $\pi$ has $k$ different slopes if and only if $\phi_\lambda|_{[0,1]}$ has $k$ different slopes.



\smallbreak

\noindent\emph{Part (ii)}. 
If $sb>1$ and $\lambda=\frac{1}{s}$, then $\phi_\lambda$ is also continuous piecewise linear with only 2-slope values, and  $\phi_\lambda(x)=0$ for $x\in[0,\frac{\epsilon}{b}]$. Suppose $(x,y,x+y)$ is an additive vertex, i.e., $\pi(x)+\pi(y)=\pi(x+y)$. Then $(\frac{x}{b},\frac{y}{b},\frac{x+y}{b})$ is an additive vertex. The additive faces of a
certain polyhedral complex $\Delta\P$ of $\phi_\lambda$, defined in analogy to the 
Gomory--Johnson case in \cite{hong-koeppe-zhou:software-abstract-with-link-to-software},  are just a scaling of those for $\pi$ \ReviewOnlyText{(see Appendix \ref{s:definition})}. The Gomory--Johnson 2-Slope Theorem for $\pi$ in \cite{infinite} guarantees that there are only 2 covered components for $\phi_\lambda$. Assume $\phi_\lambda=\frac{\phi_1+\phi_2}{2}$, then $\phi_1$ and $\phi_2$ have slope 0 wherever $\phi_\lambda$ has slope 0. From the above facts we can conclude $\phi_1=\phi_2$. Thus, $\phi_\lambda$ is extreme.

We assume $b>3$. If all intervals are covered for the restriction $\phi_\lambda|_{[0,1]}$, then we can use the same arguments to show $\phi_\lambda|_{[0,1]}$ is extreme. So we only need to show all intervals are covered by additive faces in the region: $R= \{(x,y)\colon x,y,x+y\in[0,1]\}$. Maximality of $\phi_\lambda|_{[0,1]}$ implies that if $(x,y,x+y)$ is an additive vertex, so is $(1-x-y,y,1-x)$. The fact implies that the covered components are symmetric about $x=\frac{1}{2}$, i.e., $x$ is covered $\Leftrightarrow$ $1-x$ is covered. From the scaling of additive faces of $\pi$, the additive faces of $\phi_\lambda|_{[0,1]}$ contained in the square $[0,\frac{1}{b}]^2$ cover the interval $[0,\frac{1}{b}]$. Similarly, we can use additive faces contained in $\ceil{\frac{b}{2}}$ such whole squares to cover the interval $[0,\frac{1}{2}]$. $b>3$ guarantees that those $\ceil{\frac{b}{2}}$ whole squares are contained in the region~$R$. Together with the symmetry of covered components, we can conclude all intervals are covered, thus  $\phi_\lambda|_{[0,1]}$ is extreme.
\end{proof}



\section{Computer-based search}
\label{s:com}

One of our goals is to use the computer to verify whether a given piecewise
linear function $\phi$ is a classical maximal or extreme DFF. Our technique is
analogous to that in
\cite{hong-koeppe-zhou:software-abstract-with-link-to-software}. The code
\sage{maximality\_test($\phi$)} implements a fully automatic test whether
$\phi$ is maximal, by checking the characterization of maximality for
classical DFFs given in
\cite{alves-clautiaux-valerio-rietz-2016:dual-feasible-book}.  The key
technique in the extremality test is to analyze the additivity relations in
$\Delta\P$. The foundation of the technique is that all superadditivity
conditions that are tight (satisfied with equality) for $\phi$ are also tight
for an effective perturbation $\tilde{\phi} = \phi_1-\phi = \phi - \phi_2$. We investigate the
additivity relations from additive faces of $\Delta\P$ and apply the Interval
Lemma \cite{igp_survey} and other techniques from
\cite{hong-koeppe-zhou:software-abstract-with-link-to-software} to derive
necessary properties of $\tilde{\phi}$. If $\tilde{\phi}$ is forced to be
zero, then $\phi$ is proven to be extreme \ReviewOnlyText{(see Appendix
  \ref{s:extremality})}.

We transfer the
computer-based search technique in \cite{koeppe-zhou:extreme-search} for
Gomory--Johnson functions to DFFs. Our goal is to find piecewise linear
extreme classical DFFs with rational breakpoints, which have fixed common
denominator $q\in\mathbb{N}$. The strategy is to discretize the interval
$[0,1]$ and define discrete functions on $\frac{1}{q}\mathbb{Z}\cap
[0,1]$. After adding the inequalities from characterization of maximality in
\cite{alves-clautiaux-valerio-rietz-2016:dual-feasible-book}, the space of
functions becomes a convex polytope with finite dimensions. Extreme points of
the polytope can be found by vertex enumeration tools.
Recent advances in polyhedral computation (Normaliz, version 3.2.0) allow us
to reach $q=31$ in under a minute of CPU time. 
Candidates for extreme DFFs $\phi$ are obtained by interpolating values on
$\frac{1}{q}\mathbb{Z}\cap [0,1]$ from each extreme point (discrete
function). 
Then we use our extremality test to filter out the non-extreme 
DFFs.  For example, for $q=31$, among 91761 functions
interpolated from extreme points, there are 1208 extreme DFFs, most of which
do not belong to known families.\ReviewOnlyText{{} Details can be found in
  Appendix \ref{s:search}.}

We observe most of continuous extreme DFFs are 2-slope functions by
computer-based search. In contrast to the Gomory--Johnson 2-slope Theorem
\cite{infinite}, not all 2-slope maximal classical DFFs are extreme.  Using
our computer-based search for $q=28$, we find a continuous 2-slope extreme DFF
with 3 ``covered components''
\cite{hong-koeppe-zhou:software-abstract-with-link-to-software}. Consequently
the technique for proving Gomory--Johnson 2-slope Theorem no longer works in
the DFF setting.

\clearpage
\providecommand\ISBN{ISBN }
\bibliographystyle{../../../amsabbrvurl}
\bibliography{../../../bib/MLFCB_bib}

\clearpage
\appendix

\section{Literature review on Dual-Feasible Functions}

\label{s:preliminaries}
\begin{definition}[{\cite[Definition 2.1]{alves-clautiaux-valerio-rietz-2016:dual-feasible-book}}]
A function $\phi\colon [0,1] \to [0,1]$ is called a (valid) classical Dual-Feasible Function, if for any finite index set $I$ of nonnegative real numbers $x_i \in [0,1]$, it holds that,

$$\sum_{i\in I}x_i \le 1 \Rightarrow \sum_{i\in I}\phi(x_i) \le 1 $$

\end{definition}

In order to apply classical DFFs, all variables should stay in $[0,1]$, which is not always convenient. Generalization of DFF is necessary for certain types of problem, like vector packing problems (see section 3.5 in \cite{alves-clautiaux-valerio-rietz-2016:dual-feasible-book}).

\begin{definition}[{\cite[Definition 3.1]{alves-clautiaux-valerio-rietz-2016:dual-feasible-book}}]
A function $\phi\colon \mathbb{R} \to\mathbb{R}$ is called a (valid) general Dual-Feasible Function, if for any finite index set $I$ of real numbers $x_i \in \mathbb{R}$, it holds that,

$$\sum_{i\in I}x_i \le 1 \Rightarrow \sum_{i\in I}\phi(x_i) \le 1 $$

\end{definition}

Lueker \cite{Lueker:1983:BPI:1382437.1382833} used the classical DFFs for the first time to derive lower bounds to bin packing problems. Suppose there are in total $n$ items with weight $x_i$, and each $x_i$ is drawn uniformly from the interval $[a,b]$, where $0<a<b<1$. We want to pack all items into a minimum number of bins so that no bins have weight exceeding $1$.  Define the optimum packing ratio to be the limit, as $n\to \infty$, of the ratio of
the expected value of the number of bins used to pack
$n$ items drawn uniformly from $[a, b]$ to the expected total
size of these items. Then $E[\phi(X)]/E[X]$ is the lower bound for the optimum packing ratio, where $\phi$ is a classical DFF and $X$ is the random variable uniformly distributed in $[a,b]$. 

Vanderbeck \cite{vanderbeck2000exact} proposed a parametric family of ``discrete" DFF which could be used to generate a valid inequality which is equivalent or dominates the Chv\'atal-Gomory Cut. A function $\phi\colon \{0,1,\dots, d\} \to \{0,1,\dots, d'\}$ with $d, d' \in \mathbb{Z}_+$ is said to be a \emph{discrete DFF}, if $\sum_{i\in I}x_i \le d \Rightarrow \sum_{i\in I}\phi(x_i) \le \phi(d)=d'$ for any finite index set $I$ of nonnegative integer numbers. Any discrete DFFs can be converted into classical DFFs by generating discontinuous step functions (see Section 2.1 in \cite{alves-clautiaux-valerio-rietz-2016:dual-feasible-book}). DFFs generalize the
well-known property of the floor function 
that underlies the 
Chv\'atal-Gomory Cut.

In the monograph \cite{alves-clautiaux-valerio-rietz-2016:dual-feasible-book},  the authors explored maximality of both classical and general DFFs.  

\begin{theorem}[{\cite[Theorem 2.1]{alves-clautiaux-valerio-rietz-2016:dual-feasible-book}}]
\label{thm:maximality-classical}
A function $\phi\colon [0,1] \to [0,1]$ is a classical maximal DFF if and only if the following conditions hold:
\begin{enumerate}[(i)]
  \item[(i)] 
  $\phi$ is superadditive.
        \item[(ii)] 
    $\phi$ is symmetric in the sense $\phi(x)+\phi(1-x)=1$ 
  \item[(iii)] 
    $\phi(0)=0$
  \end{enumerate}
\end{theorem}

As for the maximality of the general DFF, so far there is no characterization for that. However, there are sufficient conditions and necessary conditions explained in \cite{alves-clautiaux-valerio-rietz-2016:dual-feasible-book}.

\begin{theorem}[{\cite[Theorem 3.1]{alves-clautiaux-valerio-rietz-2016:dual-feasible-book}}]
\label{thm:maximality-general}
Let $\phi\colon \mathbb{R} \to\mathbb{R}$ be a given function. If $\phi$ satisfies the following conditions, then $\phi$ is a maximal DFF:
\begin{enumerate}[(i)]
  \item[(i)] 
  $\phi$ is superadditive.
        \item[(ii)] 
    $\phi$ is symmetric in the sense $\phi(x)+\phi(1-x)=1$ 
  \item[(iii)] 
    $\phi(0)=0$
    \item[(iv)]
    There exists an $\epsilon>0$ such that $\phi(x)\ge 0$ for all $x\in(0, \epsilon)$
  \end{enumerate}
On the other hand, if $\phi$ is a maximal general DFF, then $\phi$ satisfies conditions $(i)$, $(ii)$ and $(iv)$.
\end{theorem}

Different approaches to construct non-trivial classical DFFs from ``simple" functions are explained in \cite{alves-clautiaux-valerio-rietz-2016:dual-feasible-book}, including convex combination and function composition.

\begin{proposition}[{\cite[Section 2.3.1] {alves-clautiaux-valerio-rietz-2016:dual-feasible-book}}]
If $\phi_1$ and $\phi_2$ are two classical maximal DFFs, then $\alpha\phi_1 +(1-\alpha)\phi_2$ is also a maximal DFF, for $0<\alpha<1$.
\end{proposition}

\begin{proposition}[{\cite[Proposition 2.3] {alves-clautiaux-valerio-rietz-2016:dual-feasible-book}}]
If $\phi_1$ and $\phi_2$ are two classical maximal DFFs, then the composed function $\phi_1(\phi_2(x))$ is also a maximal DFF.
\end{proposition}

Maximal general DFFs can also be obtained by extending a maximal classical DFF to the domain $\mathbb{R}$.

\begin{theorem}[{\cite[Proposition 3.10] {alves-clautiaux-valerio-rietz-2016:dual-feasible-book}}]
Let $\phi$ be a maximal classical DFF, then there exists $b_0\ge1$ such that for all $b>b_0$ the following function $\hat{\phi}(x)$ is a maximal general DFF.
\begin{equation*}
            \hat{\phi}(x) =
                \begin{cases}
                    b\times \floor{x}+\phi(\rm{frac}(x))   &   \text{if $x\le1$ } \\
                    1-\hat{\phi}(1-x) & \text {if $x>1$}
                \end{cases}
        \end{equation*}
\end{theorem}

\begin{theorem}[{\cite[Proposition 3.12] {alves-clautiaux-valerio-rietz-2016:dual-feasible-book}}]
Let $\phi$ be a maximal classical DFF, then there exists $b\ge1$ such that the following function $\hat{\phi}(x)$ is a maximal general DFF.
\begin{equation*}
            \hat{\phi}(x) =
                \begin{cases}
                    bx+1-b   &   \text{if $x<0$ } \\
                    bx & \text {if $x>1$}\\
                    \phi(x) & \text {if $0\le x\le1$}
                \end{cases}
        \end{equation*}

\end{theorem}

DFFs can be used to generate valid inequalities for IP problems.

\begin{theorem}[{\cite[Proposition 5.1] {alves-clautiaux-valerio-rietz-2016:dual-feasible-book}}]
If $\phi$ is a maximal general DFF and $S=\{x\in\mathbb{Z}_+^n: \sum_{j=1}^{n}a_{ij}x_j\le b_j, i=1,2,\dots,m\}$. Then for any $i$, $\sum_{j=1}^{n}\phi(a_{ij})x_j\le \phi(b_j)$ is a valid inequality.
\end{theorem}

\clearpage
\section{Relation to Y{\i}ld{\i}z--Cornu\'ejols cut-generating functions}
\label{s:Yildiz-Conrnuejols}

In the paper by Y\i{}ld\i{}z and Cornu\'ejols \cite{yildiz2016cut}, the authors consider the following generalization of the Gomory--Johnson model:

\begin{equation}
\label{e:1}
x=f+\sum_{r\in\mathbb{R}}r\, y(r)
\end{equation}
$$x\in S,\, f\notin S$$ $$y:\mathbb{R} \to \mathbb{Z}_+, \,\text{and $y$ has finite support.}$$
where $S$ can be any nonempty subset of $\mathbb{R}$.
A function $\pi\colon \mathbb{R}\to \mathbb{R}$ is a valid cut-generating function if the inequality $\sum_{r\in\mathbb{R}}\pi(r)\, y(r)\ge 1$ holds for all feasible solutions $(x,y)$ to (\ref{e:1}).

\begin{theorem}
Given a valid general DFF $\phi$, then the following function is a valid cut-generating function to the model (\ref{e:1}) where $S=\{1+f\}$:
$$\pi_\lambda(x)=\frac{x-(1-\lambda)\,\phi(x)}{\lambda}, \quad 0< \lambda < 1$$
\end{theorem}

\begin{proof}
We want to show that $\pi_\lambda$ is a a valid cut-generating function to the model (\ref{e:1}) where $S=\{1+f\}$. Suppose there is a function $y:\mathbb{R} \to \mathbb{Z}_+, \,\text{$y$ has finite support}$, and $\sum_{r\in\mathbb{R}}r\, y(r)= 1$. We want to show that:
\begin{align*}
 & \sum_{r\in\mathbb{R}} \pi_\lambda(r)\, y(r)\ge 1\quad \text{holds for $\lambda\in (0,1)$}\\
 \Leftrightarrow &   \sum_{r\in\mathbb{R}} \frac{r-(1-\lambda)\,\phi(r)}{\lambda}\, y(r)\ge 1 \\
  \Leftrightarrow &   \sum_{r\in\mathbb{R}} (r-(1-\lambda)\,\phi(r))\, y(r)\ge \lambda \\
   \Leftrightarrow &   \sum_{r\in\mathbb{R}} r\, y(r) - (1-\lambda) \sum_{r\in\mathbb{R}}\phi(r)\,y(r)\ge \lambda \\
   \Leftrightarrow & \sum_{r\in\mathbb{R}}\phi(r)\,y(r)\le 1
 \end{align*}
 The last step is derived from $\sum_{r\in\mathbb{R}}r\, y(r)= 1$ and $\phi$ is a general DFF.
\end{proof}

On the other hand, given a valid cut-generating function $\pi$ to the model (\ref{e:1}) with $S=\{1+f\}$, the function $\phi(x)=\frac{x-\lambda\, \pi(x)}{1-\lambda}$ is not necessarily a general DFF. 

\begin{example}
It is not hard to show the following function is a valid function to (\ref{e:1}) with $S=\{1+f\}$.
\begin{equation*}
            \pi(x) =
                \begin{cases}
                    5x  &   \text{if $x\ge0$ } \\
                    x & \text {if $x<0$ and $x\neq -1$} \\
                    -4 &\text {if $x=-1$}
                \end{cases}
        \end{equation*} 
 Let $\lambda=\frac{1}{2}$, and $\phi(x)=\frac{x-\lambda\, \pi(x)}{1-\lambda}$. Then the following function $\phi$ is not a general DFF, since $\phi(-1)=2>1$.
 \begin{equation*}
            \phi(x) =
                \begin{cases}
                    -3x  &   \text{if $x\ge0$ } \\
                    x & \text {if $x<0$ and $x\neq -1$} \\
                    2 &\text {if $x=-1$}
                \end{cases}
        \end{equation*} 

\end{example}

Inspired by the characterization of minimal 
cut-generating functions in the Y{\i}ld{\i}z--Cornu\'ejols model in \cite{yildiz2016cut}, we find the characterization
of maximal general DFFs missing in \cite{alves-clautiaux-valerio-rietz-2016:dual-feasible-book}.

\begin{theorem}
A function $\phi\colon \mathbb{R}\to\mathbb{R}$ is a maximal general DFF if and only if the following conditions hold:
 \begin{enumerate}[(i)]
  \item[(i)] 
  $\phi(0)=0$
      \item[(ii)] 
   $\phi$ is superadditive
   \item[(iii)] 
   $\phi(x)\ge 0$ for all $x\in \mathbb{R}_+$
   \item[(iv)]
   $\phi(r)=\inf_{k}\{\frac{1}{k}(1-\phi(1-kr)): k\in \mathbb{Z}_+\}$
    \end{enumerate}

\end{theorem}

\begin{proof}
Suppose $\phi$ is a maximal general DFF, then conditions $(i),(ii),(iii)$ hold by \autoref{thm:maximality-general}. For any $r\in\mathbb{R}$ and $k\in \mathbb{Z}_+$, $kr+(1-kr)=1\Rightarrow k\phi(r)+\phi(1-kr)\le 1$. So $\phi(r)\le \frac{1}{k}(1-\phi(1-kr))$ for any positive integer $k$, then $\phi(r)\le\inf_{k}\{\frac{1}{k}(1-\phi(1-kr)): k\in \mathbb{Z}_+\}$. 

If there exists $r_0$ such that $\phi(r_0)<\inf_{k}\{\frac{1}{k}(1-\phi(1-kr_0)): k\in \mathbb{Z}_+\}$, then define a function $\phi_1$ which takes value $\inf_{k}\{\frac{1}{k}(1-\phi(1-kr_0)): k\in \mathbb{Z}_+\}$ at $r_0$ and $\phi(r)$ if $r\neq r_0$. We claim that $\phi_1$ is a general DFF which dominates $\phi$. Given $y:\mathbb{R} \to \mathbb{Z}_+, \,\text{and $y$ has finite support}$ satisfying $\sum_{r\in\mathbb{R}}r\,y(r)\le 1$.  $\sum_{r\in\mathbb{R}}\phi_1(r)\,y(r)=\phi_1(r_0)\,y(r_0)+\sum_{r\neq r_0} \phi(r)\,y(r)$. If $y(r_0)=0$, then it is clear that $\sum_{r\in\mathbb{R}}\phi_1(r)\,y(r)\le 1$. Let $y(r_0)\in\mathbb{Z}_+$, then $\phi_1(r_0)\le \frac{1}{y(r_0)}(1-\phi(1-y(r_0)\,r_0))$ by definition of $\phi_1$, then 
\begin{equation}
\phi_1(r_0)\, y(r_0)+\phi(1-y(r_0)\,r_0)\le 1
\end{equation}
From the superadditive condition and increasing property, we get 
\begin{equation}
\sum_{r\neq r_0} \phi(r)\,y(r)\le \phi(\sum_{r\neq r_0} r\,y(r)) \le \phi(1-y(r_0)\,r_0)
\end{equation}
Combine the two inequalities, then we can conclude that $\phi_1$ is a general DFF and dominates $\phi$, which contradicts the maximality of $\phi$. Therefore, the condition $(iv)$ holds.
\\
\\
Suppose there is a function $\phi\colon \mathbb{R}\to\mathbb{R}$ satisfying all four conditions. Choose $r=1$ and $k=1$, we can get $\phi(1)\le 1$. Together with condition $(i),(ii),(iii)$, it guarantees that $\phi$ is a general DFF. Assume that there is a general DFF $\phi_1$ dominating $\phi$ and there exists $r_0$ such that $\phi_1(r_0)>\phi(r_0)=\inf_{k}\{\frac{1}{k}(1-\phi(1-kr_0)): k\in \mathbb{Z}_+$. So there exists some $k\in\mathbb{Z}_+$ such that 
\begin{align*}
& \phi_1(r_0)>\frac{1}{k}(1-\phi(1-kr_0)) \\
\Leftrightarrow & k\phi_1(r_0)+\phi(1-kr_0)>1 \\
\Rightarrow & k\phi_1(r_0)+\phi_1(1-kr_0)>1
\end{align*}
The last step contradicts the fact that $\phi_1$ is a general DFF, since $kr_0+(1-kr_0)=1$. Therefore, $\phi$ is a maximal general DFF.
\end{proof}




\clearpage

\section{Definition of discontinuous piecewise linear functions and polyhedral
  complexes underlying the algorithmic maximality test of Dual-Feasible Functions}
\label{s:definition}
In this section, we focus on classical DFFs.
We begin with a definition of piecewise linear
functions~$\phi\colon [0,1]\to[0,1]$ that are allowed to be discontinuous, similar
to \cite[section 2.1]{basu-hildebrand-koeppe:equivariant}
and~\cite{igp_survey}. 
Let $0 =a_0 < a_1< \dots < a_{n-1} < a_n = 1$.
Denote by $\B = \{ a_0, a_1, \dots, a_{n-1}, a_n \} $ the set of all
possible \emph{breakpoints}. 
The 0-dimensional faces are defined to be the 
singletons, $\{ a_i \}$, $a_i\in B$,
and the 1-dimensional faces are the closed intervals,
$ [a_i, a_{i+1}]$, $i=0, \dots, {n-1}$. 
Together they form $\P = \P_{\B}$, a finite 
polyhedral complex.
%
We call a function $\phi\colon [0,1] \to \R$ 
\emph{piecewise linear} over $\mathcal{P}_B$ if for each face $I \in
\mathcal{P}_B$, there is an affine linear function $\phi_I \colon \R \to \R$,
$\phi_I(x) = c_I x + b_I$  such that $\phi(x) = \phi_I(x)$ for all $x
\in\relint(I)$. 
Under this definition, piecewise linear functions can be discontinuous.
Let $I = [a_i, a_{i+1}]$. The function $\phi$ can be determined on the open intervals
$\intr(I) = (a_i, a_{i+1})$ by linear
interpolation of the limits $\phi(a_i^+)=\lim_{x\to a_i, x>a_i} \phi(x)
  = \phi_I(a_i)$ and $\phi(a_{i+1}^-)=\lim_{x\to a_{i+1}, x<a_{i+1}} \phi(x) = \phi_I(a_{i+1})$. 
We say the function $\phi$ is continuous piecewise linear over $\mathcal{P}_B$
if it is affine over each of the cells of $\mathcal{P}_B$ (thus automatically
imposing continuity). 

Unlike Gomory--Johnson cut-generating functions, which may be discontinuous at $0$ on both sides, a classical maximal DFF is always continuous at 0 from the right and at 1 from the left. 

\begin{lemma}
Any piecewise linear maximal classical DFF is continuous at $0$ from the right and continuous at $1$ from the left.
\end{lemma}

\begin{proof}
Consider $\phi$ to be a piecewise linear maximal classical DFF, and $\phi(x)=sx+b$ on the first open interval $(a_0, a_1)$. Note that the maximality of $\phi$ implies that $\phi(0)=0$. Choose $x=y=\frac{a_1}{3}$. Then based on superadditivity, we have 
$$\phi(x)+\phi(y)\le\phi(x+y)\Rightarrow sx+b+sy+b\le s(x+y)+b\Rightarrow b\le0$$
$b$ is also the right limit at $0$, so $b$ is nonnegative. Therefore, $b=0$, which implies $\phi$ is continuous at $0$ from the right. By symmetry, $\phi$ is continuous at $1$ from the left.
\end{proof}

Similar to \cite{basu-hildebrand-koeppe:equivariant,igp_survey}, we
introduce the function 
$\nabla\phi \colon \R \times \R \to \R$, $\nabla\phi(x,y) =
  \phi(x+y) - \phi(x) - \phi(y)$, which measures the
slack in the superadditivity condition.
The piecewise linearity of $\phi(x)$ 
induces piecewise linearity of $\nabla\phi(x,y)$.  To express the domains of
linearity of $\nabla\phi(x,y)$, and thus domains of additivity and strict
superadditivity, we introduce the two-dimensional polyhedral complex
$\Delta\P = \Delta\P_\B$. 
The faces $F$ of the complex are defined as follows. Let $I, J, K \in
\P_{\B}$, so each of $I, J, K$ is either a breakpoint of $\phi$ or a closed
interval delimited by two consecutive breakpoints. Then 
$ F = F(I,J,K) = \setcond{\,(x,y) \in \R \times \R}{x \in I,\, y \in J,\, x + y \in
  K\,}$. 
The projections  $p_1, p_2, p_3 \colon \R \times \R \to \R$ are defined as 
$p_1(x,y)=x$,  $p_2(x,y)=y$, $p_3(x,y) = x+y$.
Let $F \in \Delta\P$ and let $(u,v) \in F$. Observe that the piecewise linearity of $\phi$ induces piecewise
linearity of $\nabla \phi$, thus $\nabla\phi|_{\relint(F)}$ is affine, we define 
\[\nabla\phi_F(u,v) = \lim_{\substack{(x,y) \to (u,v)\\ (x,y) \in \relint(F)}}
  \nabla\phi(x, y),\]
which allows us to conveniently express limits to boundary points of $F$, in
particular to vertices of $F$, along paths within $\relint(F)$. 
It is clear that $\nabla\phi_F(u,v) $ is affine over $F$, and $\nabla\phi(u,v)=\nabla\phi_F(u,v)$ for all $(u,v) \in \relint(F)$.
We will use $\verts(F)$ to denote the set of vertices of the face~$F$.  

Let $\phi$ be a piecewise linear maximal DFF. We now define the
\emph{additive faces} of the two-dimensional polyhedral complex $\Delta\P$ of
$\phi$. When $\phi$ is continuous, we say that a face $F \in \Delta\P$
is additive if $\nabla\phi =0$ over all $F$. Notice that $\nabla\phi$ is affine
over $F$, the condition is equivalent to $\nabla\phi(u, v) = 0$ for any $(u, v)
\in \verts(F)$. When $\phi$ is discontinuous, following
\cite{hong-koeppe-zhou:software-abstract-with-link-to-software}, we say that a face $F \in
\Delta\P$  is additive if $F$ is contained in a face $F' \in \Delta\P$ such
that $\nabla\phi_{F'}(x,y) =0$ for any $(x,y) \in F$.
Since $\nabla\phi$ is affine in the relative interiors of each face of
$\Delta\P$, the last condition is equivalent to $\nabla\phi_{F'}(u,v) =0$ for
any $(u,v) \in \verts(F)$. 

One of our goals is to use the computer to verify whether a given function, which is assumed to be piecewise linear, is a classical maximal or extreme DFF. In terms of maximality, the two main conditions we need to check are the superadditivity and the symmetry condition. In order to check the superadditivity and the symmetry condition on the whole interval $[0,1]$, we only need to check on all possible breakpoints including the limit cones, which should be a finite set. As for extremality, we use the similar technique in  \cite{basu-hildebrand-koeppe:equivariant,igp_survey} to try to find equivariant perturbation or finite dimensional perturbation.

We introduce an efficient method to check the maximality of a given piecewise linear function using the computer. The code \sage{maximality\_test($\phi$)} implements a fully automatic test whether a given function $\phi$ is maximal, by using the information that is described in additive faces in $\Delta\P$. 

Based on Theorem \ref{thm:maximality-classical}, we need to first check that the range of the function stays in $[0,1]$ and $\phi(0)=0$. Since we assume the function is piecewise linear with finitely many breakpoints, only function values and left/right limits at the breakpoints need to be checked. Similarly, the symmetry condition only needs to checked on the set of breakpoints of $\phi$, namely $\B$, including the left and right limits at each breakpoint. In regards to the superadditivity, it suffices to check $\nabla\phi(u,v)\ge 0$ for any $(u,v)\in \verts(F)$, including the limit values $\nabla \phi_F(u,v)$ when $\phi$ is discontinuous. 

As for the diagrams of $\Delta\P$, we start with a triangle complex $I=J=K=[0,1]$, and then refine $I,J,K$ based on  the set of breakpoints $\B$. In practice, the code \sage{maximality\_test($\phi$)} will show vertices where superadditivity or symmetry condition is violated, and it will paint 2-dimensional additive faces green. It also marks 1-dimensional and 0-dimensional additive faces, which are additive edges and vertices not contained in any higher dimensional additive faces.

\autoref{fig:2d_diagram} is an example of a maximal DFF. We show the diagram of $\Delta\P$ with additive faces painted green, and we also show the function on the upper and left borders. There is no vertex where superadditivity or symmetry condition is violated, so the function is maximal. 

\begin{figure}[tp]
  \centering\includegraphics[width=.7\linewidth]{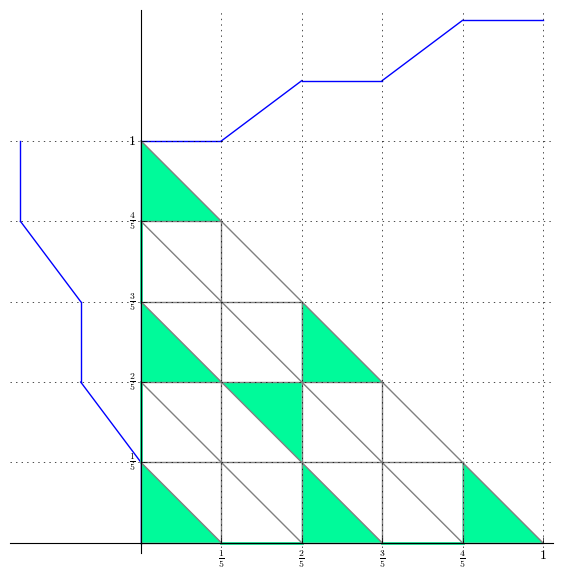}
\caption{Maximal DFF $\phi_{BJ,1}(x;C)=\frac{\floor{Cx}+\max(0,\frac{\{Cx\}-\{C\}}{1-\{C\}})}{\floor{C}}$ for $C=\frac{5}{2}$.}
\label{fig:2d_diagram}
\end{figure}

\clearpage

\section{Detailed proof of Theorem \autoref{thm:main}}
\label{s:proof}

\begin{proof}
We cite two theorems for proving maximality of DFFs. \autoref{thm:maximality-classical} is the characterization of maximal classical DFFs and \autoref{thm:maximality-general} contains sufficient conditions and necessary conditions for maximal general DFFs.

First we prove $\phi_\lambda$ is a maximal general DFF if $\lambda$ is small enough. As a minimal valid Gomory--Johnson function, $\pi$ is $\mathbb{Z}$-periodic, $\pi(0)=0$, $\pi$ is subadditive and $\pi(x)+\pi(b-x)=1$ for all $x\in \mathbb{R}$ \cite{igp_survey}. Note that $\phi_\lambda$ is defined on $\mathbb{R}$, since $\pi$ is $\mathbb{Z}$-periodic and defined on $\mathbb{R}$. It is not hard to check $\phi_\lambda(0)=0$. Since $\phi_\lambda$ is obtained by subtracting a subadditive function from a linear function, it is superadditive.
\begin{align*}
\phi_\lambda(x)+\phi_\lambda(1-x) & = \frac{bx-\lambda\pi(bx)}{b-\lambda}+\frac{b(1-x)-\lambda\pi (b(1-x))}{b-\lambda}
\\& =\frac{b-\lambda (\pi(bx)+\pi(b(1-x)))}{b-\lambda}=1
\end{align*}
The last step is from the symmetry condition of $\pi$ and $\pi(b)=1$.
Since $\pi$ is piecewise linear and continuous at $0$ from the right. Let $s$ be the largest slope of $\pi$, then the largest slope of $\pi(bx)$ is $bs$. Choose $\delta=\frac{1}{s}$, then if $\lambda<\delta$, the slope of $bx$ is always no smaller than the slope of $\lambda \pi(bx)$. There exists an $\epsilon>0$ such that $\phi_\lambda(x)\ge 0$ for all $x\in(0, \epsilon)$. Therefore, $\phi_\lambda$ is a general maximal DFF by \autoref{thm:maximality-general} and $\phi_\lambda|_{[0,1]}$ is a maximal classical DFF by \autoref{thm:maximality-classical}. 

\smallbreak

\noindent\emph{Part (i)}. 
Suppose $\pi$ has slope $s$ on the interval $(a_i, a_{i+1})$, then by calculation $\phi_\lambda(x)=\frac{bx-\lambda\pi(bx)}{b-\lambda}$ has slope $s'=\frac{b(1-\lambda s)}{b-\lambda}$ on the interval $(\frac{a_i}{b},\frac{a_{i+1}}{b})$. So if $\pi$ has slope $s_1$, $s_2$ on interval $(a_i, a_{i+1})$ and $(a_j, a_{j+1})$ respectively, and $\phi_\lambda$ has slope $s_1'$, $s_2'$ on interval $(\frac{a_i}{b},\frac{a_{i+1}}{b})$ and $(\frac{a_j}{b},\frac{a_{j+1}}{b})$ respectively, then $s_1=s_2$ if and only if $s_1'=s_2'$. From the above fact we can conclude $\pi$ has $k$ different slopes if and only if $\phi_\lambda$ has $k$ different slopes.  

Since $\pi$ is $\mathbb{Z}$-periodic, $\phi_\lambda$ is quasiperiodic with period $\frac{1}{b}$. If $b>1$, the interval $[0,1]$ contains a whole period, which has pieces with all different slope values. So $\pi$ has $k$ different slopes if and only if $\phi_\lambda|_{[0,1]}$ has $k$ different slopes.
\smallbreak

\noindent\emph{Part (ii)}. 
%
%
%
%
If $sb>1$ and $\lambda=\frac{1}{s}$, then it is not hard to show $\phi_\lambda$ is also continuous piecewise linear with only 2-slope values, and  $\phi_\lambda(x)=0$ for $x\in[0,\frac{\epsilon}{b}]$, i,e., one slope value is $0$. From the above results, we know $\phi_\lambda$ is a maximal general DFFs. 

We use the idea of extremality test in \autoref{s:extremality}. $\pi$ is extreme from the Gomory--Johnson 2-Slope Theorem \cite{infinite}, therefore all intervals are covered and there are 2 covered components. Suppose $(x,y,x+y)$ is an additive vertex, which means $\pi(x)+\pi(y)=\pi(x+y)$. From arithmetic computation, $(\frac{x}{b},\frac{y}{b},\frac{x+y}{b})$ is an additive vertex, i.e., $\phi_\lambda(\frac{x}{b})+\phi_\lambda(\frac{y}{b})=\phi_\lambda(\frac{x+y}{b})$. So the additive faces for $\phi_\lambda$ are just a scaling of those for $\pi$. In regards to $\phi_\lambda$, all intervals are covered and there are only 2 covered components. $\phi_\lambda(1)=1$ and $\phi_\lambda(x)=0$ for $x\in[0,\frac{\epsilon}{b}]$ guarantee that the interval $[0,1]$ contains the 2 covered components. 

Assume $\phi_\lambda=\frac{\phi_1+\phi_2}{2}$, where $\phi_1$ and $\phi_2$ are maximal general DFFs. By \autoref{thm:maximality-general} and definition, $\phi_1(x)=\phi_2(x)=0$ for $x\in[0,\frac{\epsilon}{b}]$ and  $\phi_1(1)=\phi_2(1)=1$. $\phi_1$ and $\phi_2$ satisfy the additivity where $\phi_\lambda$ satisfies the additivity, otherwise one of $\phi_1$ and $\phi_2$ violates the superadditivity. So the additive faces of $\phi_\lambda$ are still additive faces of $\phi_1$ and $\phi_2$. By Interval Lemma \cite{igp_survey} and values at point $\frac{\epsilon}{b}$ and $1$, we can show $\phi_1$ and $\phi_2$ both have 2 covered components and these covered components are the same as those of $\phi_\lambda$. Thus $\phi_1$ and $\phi_2$ are both continuous 2-slope functions and one slope value is 0, due to nondecreasing condition. Suppose the 2 covered components within $[0,1]$ are $C_1$ and $C_2$, where $C_1$ and $C_2$ are disjoint unions of closed intervals. We assume $\phi_1$ and $\phi_2$ have slope $0$ on $C_1$ and slope $s_1$ and $s_2$ on $C_2$ respectively. $\phi_1(1)=\phi_2(1)=1$ implies that $0\times |C_1|+s_1\times |C_2|=1$ and $0\times |C_1|+s_2\times |C_2|=1$, where $ |C_1|$ and $ |C_2|$ denote the measure of $C_1$ and $C_2$. So we have $s_1=s_2$. All these properties guarantee that $\phi_1$ and $\phi_2$ are equal to each other, therefore $\phi_\lambda$ is extreme. 

We assume $b>3$. If all intervals are covered for the restriction
$\phi_\lambda|_{[0,1]}$, then we can use the same arguments to show
$\phi_\lambda|_{[0,1]}$ is extreme. So we only need to show all intervals are
covered by additive faces in the triangular region: $R= \{(x,y)\colon x,y,x+y\in[0,1]\}$. Maximality of $\phi_\lambda|_{[0,1]}$, especially the symmetry condition, implies that if $(x,y,x+y)$ is an additive vertex, so is $(1-x-y,y,1-x)$. The fact implies that the covered components are symmetric about $x=\frac{1}{2}$, i.e., $x$ is covered $\Leftrightarrow$ $1-x$ is covered and they are in the same covered components. From the scaling of additive faces of $\pi$, the additive faces of $\phi_\lambda|_{[0,1]}$ contained in the square $[0,\frac{1}{b}]^2$ cover the interval $[0,\frac{1}{b}]$, and the additive faces of $\phi_\lambda|_{[0,1]}$ contained in the square $[\frac{1}{b},\frac{2}{b}]\times[0,\frac{1}{b}]$ cover the interval $[\frac{1}{b},\frac{2}{b}]$. Similarly, we can use additive faces contained in $\ceil{\frac{b}{2}}=\ceil{\frac{1}{2}/\frac{1}{b}}$ such whole squares to cover the interval $[0,\frac{1}{2}]$. $b>3$ guarantees that those $\ceil{\frac{b}{2}}$ whole squares are contained in the region $R$. Together with the symmetry of covered components, we can conclude all intervals are covered, thus  $\phi_\lambda|_{[0,1]}$ is extreme.

This concludes the proof of the theorem.
\end{proof}

\clearpage
\section{Extremality test}
\label{s:extremality}

In this section, we explore extremality test for a given function and ways to construct perturbation functions. First there is a simple necessary condition for piecewise linear classical extreme DFFs. 

\begin{lemma}
Let $\phi$ be a piecewise linear classical extreme DFF. If $\phi$ is strictly increasing, then $\phi(x)=x$. In other words, there is no strictly increasing piecewise linear classical extreme DFF except for $\phi(x)=x$.
\end{lemma}

\begin{proof}
We know $\phi$ is continuous at $0$ from the right. Suppose $\phi(x)=sx$, $x\in[0,a_1)$ and $s>0$. $\phi$ is not strictly increasing if $s=0$. In order to satisfy the superadditivity, $s$ should be the smallest slope value, which implies $s\le1$ since $\phi(1)=1$. Similarly if $s=1$, then $\phi(x)=x$.

Next, we can assume $0<s<1$. Define a function: $$\phi_1(x)=\frac{\phi(x)-sx}{1-s}$$
It is not hard to show $\phi_1(x)=0$ for $x\in [0,a_1)$, and $\phi_1(1)=1$. $\phi_1$ is superadditive because it is obtained by subtracting a linear function from a superadditive function. These two together guarantee that $\phi_1$ stays in the range $[0,1]$.
$$\phi_1(x)+\phi_1(1-x)=\frac{\phi(x)+\phi(1-x)-sx-s(1-x)}{1-s}=1$$
The above equation shows that $\phi_1$ satisfies the symmetry condition. Therefore, $\phi_1$ is also a maximal classical DFF. $\phi(x)=sx+(1-s)\phi_1(x)$ implies $\phi$ is not extreme, since it can be expressed as a convex combination of two different maximal DFFs: $x$ and $\phi_1$.
\end{proof}

Next we give the definition of the effective perturbation function.

\begin{definition}
Let $\phi$ be a maximal classical DFF. Then a function $\tilde{\phi}\colon [0,1] \to \mathbb{R}$ is called an effective perturbation function of $\phi$, if there exists $\epsilon >0$, such that $\phi+\epsilon \tilde{\phi}$ and $\phi-\epsilon \tilde{\phi}$ are both maximal DFFs.
\end{definition}

Effective perturbations of a DFF $\phi$ have a close relation to the functions $\phi$ in regards to continuity and superadditivity.

\begin{lemma}
Let $\phi$ be a piecewise linear maximal classical DFF. If $\phi$ is continuous on a proper interval $I\subseteq [0,1]$, then for any perturbation function $\tilde{\phi}$, we have that $\tilde{\phi}$ is Lipschitz continuous on the interval $I$. Furthermore, $\tilde{\phi}$ is continuous at all points at which $\phi$ is continuous.
\end{lemma}

\begin{proof}
We know $\phi$ is continuous at $0$ from the right. Let $\tilde{\phi}$ to be an effective perturbation function. Since $\phi$ is piecewise linear, there exists a nonnegative $s$, such that 
$\phi(x)=sx$ on the first interval $[0, x_1)$. Let $I=J=K=[0,x_1]$, and let $F=F(I,J,K)$. Then for any $x\in I$, $y\in J$, $x+y\in K$, $\nabla\phi_F(x,y)=s(x+y)-sx-sy=0$. Thus, $F$ is a
two-dimensional additive face of $\Delta P$. From the Interval Lemma, we know that there exists $\tilde{s}$, such that $\tilde{\phi}(x)=\tilde{s}x$, when $x\in[0,x_1)$. Since $\tilde{\phi}$ is an effective perturbation function, there exists $\epsilon>0$, such that $\phi^+=\phi+\epsilon\tilde{\phi}$ and $\phi^-=\phi-\epsilon\tilde{\phi}$ are both maximal DFFs. We know that $\phi^+$ and $\phi^-$ have slope $s^+=s+\epsilon\tilde{s}\ge 0$ and $s^-=s-\epsilon\tilde{s}\ge 0$ respectively.

Let $I\subseteq [0,1]$ be a proper interval where $\phi$ is continuous. Since $\phi$ is piecewise linear, there exists a positive constant $C$ such that $|\phi(x)-\phi(y)|\le C|x-y|$, for any $x,y \in I$. We can simply choose $C$ to be the largest absolute values of the slopes of $\phi$. Assume $x\ge y$ and $x-y<x_1$, from the superadditivity of $\phi^+$ and $\phi^-$, $\phi^+(x)\ge \phi^+(y)+\phi^+(x-y)=\phi^+(y)+s^+(x-y)$ and $\phi^-(x)\ge \phi^-(y)+\phi^-(x-y)=\phi^-(y)+s^-(x-y)$. It follows that $-(C+s^-)(x-y)\le \epsilon(\tilde{\phi}(x)-\tilde{\phi}(y))\le (C+s^+)(x-y)$. Therefore, $|\tilde{\phi}(x)-\tilde{\phi}(y)|\le \tilde{C} |x-y|$, where $\tilde{C}=\frac{1}{\epsilon} \max(C+s^-, C+s^+)$. $\tilde{\phi}$ is Lipschitz continuous on the interval $I$.
\end{proof}

\begin{lemma}
Let $\phi$ be a piecewise linear maximal classical DFF. For any effective perturbation function $\tilde{\phi}$, we have that $\tilde{\phi}$ satisfies additivity where $\phi$ satisfies additivity.
\end{lemma}

\begin{proof}
Since $\tilde{\phi}$ is an effective perturbation function, there exists $\epsilon>0$, such that $\phi^+=\phi+\epsilon\tilde{\phi}$ and $\phi^-=\phi-\epsilon\tilde{\phi}$ are both maximal DFFs. If $\phi$ satisfies additivity at $(x,y)$, meaning $\phi(x)+\phi(y)=\phi(x+y)$. Applying superadditivity of $\phi^+$ and $\phi^-$ at $(x,y)$, we get $\tilde{\phi}(x)+\tilde{\phi}(y)=\tilde{\phi}(x+y)$.
\end{proof}

Similar to \cite{basu-hildebrand-koeppe:equivariant,igp_survey}, we can find 2-dimensional additive faces and project in 3 directions to get covered intervals and uncovered intervals.

If there is some uncovered interval, our code can construct a nontrivial effective equivariant perturbation function, using the same technique in \cite{basu-hildebrand-koeppe:equivariant,igp_survey}. Thus, extremality test returns false.

If $[0,1]$ is covered by $C_1,\dots,C_k$, each $C_i$ is a connected covered interval. By Interval Lemma, we know $\phi$ and $\tilde{\phi}$ are affine linear on each $C_i$ with the same slope. Therefore, we have $k$ slope variables $s_1,\dots,s_k$. Between each pair of adjacent intervals, there may exists a jump, where $\phi$ is discontinuous. So we also need to introduce $m$ jump variables. One can use the functionality of piecewise linear functions to define $g\colon[0,1]\to {\mathbb{R}}^{k+m}$ so that $\tilde{\phi}(x)=g(x)\cdot (s_1,\dots,s_k, h_1,\dots,h_m)$. 

The next step is to find all constraints $\tilde{\phi}(x)$ needs to satisfy and solve a linear system of $(s_1,\dots,s_k, h_1,\dots,h_m)$. If there is only the trivial solution, then there is no finite dimensional perturbation. If one nonzero function $\tilde{\phi}(x)$ is found, then a positive $\epsilon$ can be found by our code.

Using the following lemma can simplify the extremality test.

\begin{lemma}
Let $\phi$ be a piecewise linear maximal classical DFF, then $0\le\phi(x_1^-)\le x_1$. If $0<\phi(x_1^-)<x_1$, then it is not extreme. If $\phi(x_1^-)=x_1$, then $\phi(x)=x$, thus extreme.
\end{lemma}

\begin{lemma}
Let $\phi$ be a piecewise linear maximal classical DFF, and $\tilde{\phi}$ be a function: $[0,1] \to \mathbb{R}$. Assume $\phi(x_1^-)=0$, and there are no uncovered interval. Let $\hat{B}$ be the union of breakpoints of $\phi$ and $\tilde{\phi}$, $\P$ be the new complex based on $\hat{B}$. Consider all vertices $(x,y)$, including limit cones in discontinuous case, on the new complex where $\phi$ satisfies additivity, and construct a linear system of equations for $\tilde{\phi}$: $\nabla\tilde{\phi}(x,y)=0$, $\tilde{\phi}(1)=0$ and $\tilde{\phi}(x_1^-)=0$. If there is only the trivial solution $\tilde{\phi}(x)=0$, then $\phi$ is extreme. If there is some nontrivial solution $\tilde{\phi}$, then there exists $\epsilon>0$ such that $\phi+\epsilon \tilde{\phi}$ and $\phi-\epsilon \tilde{\phi}$ maximal, thus $\phi$ is not extreme.
\end{lemma}

\begin{proof}
Note that if $\tilde{\phi}$ is an effective perturbation function, then it must satisfy $\nabla\tilde{\phi}(x,y)=0$ and $\tilde{\phi}(1)=0$. If $[0,1]$ is covered by $C_1,\dots, C_k$, each $C_i$ is a connected covered components. By Interval Lemma, we know $\phi$ and $\tilde{\phi}$ are affine linear on each $C_i$ with the same slope. It is clear that $\tilde{\phi}=0$ if the linear system has only the trivial solution, thus $\phi$ is extreme.

Suppose there is a nonzero solution $\tilde{\phi}$ to the linear system, and $\tilde{\phi}$ should also be a piecewise linear function on $[0,1]$. Denote $M$ to be the largest absolute value of $\tilde{\phi}$, and $s$ to be the largest absolute slope value of $\tilde{\phi}$. Let $V$ be the set of vertices, including limit cones in the discontinuous case, where $\tilde{\phi}$ satisfy strict superadditivity, which means $\nabla\tilde{\phi}(x,y)>0$. Since we are restricted to piecewise linear functions, $V$ is a finite set. Let $\delta=\min\{\nabla\phi(x,y)|(x,y)\in V\}>0$, and $\sigma=\max\{|\nabla\tilde{\phi}(x,y)||(x,y)\in V\}> 0$. Choose $\epsilon=\frac{\delta}{\sigma}>0$, then claim that $\phi+\epsilon\tilde{\phi}$ and $\phi-\epsilon \tilde{\phi}$ are both superadditive. Compute $ \nabla(\phi+\epsilon\tilde{\phi})|_{(x,y)}=\nabla\phi|_{(x,y)}+\epsilon\nabla\tilde{\phi}|_{(x,y)}$. If $(x,y)\in \hat{B}$, then both $\nabla\phi|_{(x,y)}$ and $\nabla\tilde{\phi}|_{(x,y)}$ are $0$. Otherwise $(x,y)\in V$, then $\nabla\phi|_{(x,y)}+\epsilon\nabla\tilde{\phi}|_{(x,y)}\ge\delta-\epsilon\sigma=0$.

Consider $\phi$ and $\tilde{\phi}$ on the interval $[x_0,x_1)$, they are both linear with slope 0, since $\phi(x_1^-)=0$ and $\tilde{\phi}(x_1^-)=0$. Then $\phi+\epsilon \tilde{\phi}$ and $\phi-\epsilon \tilde{\phi}$ are both  nonnegative on $[x_0,x_1)$ and superadditive, so they are increasing and stay in the range of $[0,1]$.

Symmetry condition of $\phi+\epsilon \tilde{\phi}$ and $\phi-\epsilon \tilde{\phi}$ is implied by the symmetry condition of $\phi$. Every $(x,1-x)$ in the complex is an additive vertex, then we get the symmetry condition from the linear system $\nabla\tilde{\phi}(x,1-x)=0$.

Therefore, both $\phi+\epsilon \tilde{\phi}$ and $\phi-\epsilon \tilde{\phi}$ are maximal DFFs, thus $\phi$ is not extreme.
\end{proof}

\begin{lemma}
Suppose $\phi_1$ and $\phi_2$ are two maximal classical DFFs, and $\phi_1(\phi_2(x))$ is an extreme DFF. If $\phi_2$ is continuous, then $\phi_1$ must be extreme.
\end{lemma}

\begin{proof}
Suppose $\phi_1$ is not extreme, which implies there exist two different maximal DFFs $\hat{\phi_1}$ and $\hat{\phi_2}$ such that $\phi_1=\frac{1}{2}(\hat{\phi_1}+\hat{\phi_2})$. Then 

$$\phi_1(\phi_2(x))=\frac{1}{2}(\hat{\phi_1}(\phi_2(x)+\hat{\phi_2}(\phi_2(x))$$

Both $\hat{\phi_1}(\phi_2(x))$ and $\hat{\phi_1}(\phi_2(x))$ are maximal DFFs, and the above equality shows that $\phi_1(\phi_2(x))$ can be expressed as a convex combination of $\hat{\phi_1}(\phi_2(x))$ and $\hat{\phi_1}(\phi_2(x))$. To find a contradiction, we only need to show $\hat{\phi_1}(\phi_2(x))$ and $\hat{\phi_1}(\phi_2(x))$ are two different functions. The range of $\phi_2$ is exactly $[0,1]$ due to maximality and continuity of $\phi_2$. Since $\hat{\phi_1}$ and $\hat{\phi_2}$ are distinct and the range of $\phi_2$ is $[0,1]$, $\hat{\phi_1}(\phi_2(x))$ and $\hat{\phi_1}(\phi_2(x))$ are distinct. Therefore, $\phi_1(\phi_2(x))$ is not extreme, which is a contradiction. So $\phi_1$ must be extreme.
\end{proof}

\smallbreak

\clearpage
\section{Computer-based search}
\label{s:search}

In this section, we discuss how
computer-based search can help in finding extreme classical DFFs. Most known
classical DFFs in the monograph
\cite{alves-clautiaux-valerio-rietz-2016:dual-feasible-book} have a similar
structure. The known continuous DFFs are 2-slope functions, and known
discontinuous DFFs have slope 0 in every affine linear piece. By using the
Normaliz and PPL, we have found many more new extreme classical DFFs.  

\subsection{PPL and Normaliz}

Based on a detailed computational study regarding the performance of vertex
enumeration codes in \cite{koeppe-zhou:extreme-search}, 
we consider two libraries, the Parma Polyhedra Library (PPL) and Normaliz.  
Both are convenient to use within the software SageMath \cite{sage}. 

\subsection{Functions on a grid}

First, our goal is to find all piecewise linear extreme DFFs, both continuous and discontinuous with rational breakpoints with fixed common denominator $q\in\mathbb{N}$. We use $B_q$ to denote the set $\{0, \frac{1}{q}, \frac{2}{q},\dots,\frac{q-1}{q},1\}$

\begin{definition}
Denote $\Phi_C(q)$ to be the set of all maximal continuous piecewise linear DFFs with breakpoints in $B_q$, and $\Phi_D(q)$ to be the set of all maximal discontinuous piecewise linear DFFs with breakpoints in $B_q$.
\end{definition}

\begin{theorem}
Both $\Phi_C(q)$ and $\Phi_D(q)$ are finite dimensional convex polytopes, if $q$ is fixed.
\end{theorem}

\begin{proof}
Note that any $\phi\in\Phi_C(q)$ is uniquely determined by the values at the breakpoints. So we just need to consider discrete functions on $B_q$, or the restriction of $\phi$: $\phi|_{B_q}$. Since $\phi$ is maximal, $\phi|_{B_q}$ should also stays in the range $[0,1]$, and satisfy superadditivity and symmetry condition.

For each possible breakpoint $\frac{i}{q}$, we introduce a variable $a_i$ to be the value at $\frac{i}{q}$. After adding the inequalities from superadditivity and symmetry condition and $0\le a_i \le 1$, $a_0=0$, we will get a polytope in $q+1$ dimensional space, because there are only finitely many inequalities and each variable is bounded.

It is not hard to prove the convex combination of two maximal continuous piecewise linear DFFs with breakpoints in $B_q$ is also in $\Phi_C(q)$.

We can get $\phi$ back by interpolating $\phi|_{B_q}$. Therefore  $\Phi_C(q)$ is a finite dimensional convex polytope. Similarly we can prove for discontinuous case, since we just need to add two more variables for the left and right limits at each possible breakpoint $\frac{i}{q}$.
\end{proof}

\begin{definition}
Denote $\Phi'_C(q)$ to be the set of all discrete functions $\phi$ on $B_q$ which satisfy superadditivity, symmetry condition and $\phi(0)=0$. Denote $\Phi'_D(q)$ to be the set of all discrete functions $\phi$ on the grid: $$B'_q=\{0, \tfrac{1}{q}^-, \tfrac{1}{q}, \tfrac{1}{q}^+, \tfrac{2}{q}^-, \tfrac{2}{q}, \tfrac{2}{q}^+,\dots, \tfrac{q-1}{q}^-, \tfrac{q-1}{q}, \tfrac{q-1}{q}^+, 1\}$$ which satisfy superadditivity, symmetry condition and $\phi(0)=0$.
\end{definition}

As we can see in the above proof, the polytope of continuous functions and that of discrete functions are the same polytope. Continuous functions and discrete functions have a bijection by restriction and interpolation. Therefore, we have $\Phi_C(q)\cong \Phi'_C(q)$ and $\Phi_D(q)\cong \Phi'_D(q)$.

To summarize, the strategy is to discretize the interval $[0,1]$ and define discrete functions on $B_q$. After adding the inequalities from superadditivity and symmetry condition, the space of functions becomes a convex polytope. Extreme points of the polytope can be found by Normaliz or PPL. The possible DFF $\phi$ is obtained by interpolating values on $B_q$ from each extreme point. Under such hypotheses, $\phi$ is uniquely determined by its values on $B_q$. Discontinuous functions can also be found in this way, since we just need to add more variables for the left and right limits at all possible breakpoints.

So far the function $\phi$ is only a candidate of extreme DFF, since extremality of the discrete function does not always imply extremality of the continuous function by interpolation. We need to use extremality test described in Appendix \ref{s:extremality} to pick those actual extreme DFFs. The following theorem provides an easier verification for extremality: if $\phi$ has no uncovered interval, then we can claim we find an extreme DFF.

\begin{theorem}
Let $\phi$ be a function from interpolating values of some extreme point of the polytope $\Phi'_C(q)$ or $\Phi'_D(q)$. If there is no uncovered interval, then $\phi$ is extreme.
\end{theorem}

\begin{proof}
Here we only give the proof for continuous case, and the proof for discontinuous case is similar.
Suppose $\phi$ is obtained by in interpolating the discrete function $\phi|_{B_q}$, which is an extreme point of the polytope $\Phi'_C(q)$, and $\tilde{\phi}$ is an effective perturbation function.

If there is no uncovered interval for $\phi$, then the interval $[0,1]$ is covered by $C_1,\dots, C_k$, each $C_i$ is a connected covered component. Since every breakpoint of $\phi$ is in the form of $\frac{i}{q}$, the endpoints of $C_i$ are also in the form of $\frac{i}{q}$. We know $\phi$ and $\tilde{\phi}$ are affine linear on each $C_i$ with the same slope by Interval Lemma, and continuity of $\phi$ implies continuity of $\tilde{\phi}$. Therefore, we know $\tilde{\phi}$ is also a continuous function with breakpoints in $B_q$, which means $\phi + \epsilon \tilde{\phi}$ and $\phi + \epsilon \tilde{\phi}$ both have the same property. The maximality of $\phi + \epsilon \tilde{\phi}$ and $\phi + \epsilon \tilde{\phi}$ implies their restrictions to $B_q$ are also in the polytope $\Phi_C(q)$, and 

$$\phi|_{B_q}=\frac{(\phi+\epsilon\tilde{\phi})|_{B_q}+(\phi-\epsilon\tilde{\phi})|_{B_q}}{2}$$

Since $\phi|_{B_q}$ is an extreme point of the polytope $\Phi_C(q)$, then $\phi|_{B_q}=(\phi+\epsilon\tilde{\phi})|_{B_q}=(\phi-\epsilon\tilde{\phi})|_{B_q}$, which implies $\phi=\phi+\epsilon\tilde{\phi}=\phi-\epsilon\tilde{\phi}$. Therefore, $\phi$ is extreme.

\end{proof}

\autoref{table:nonlin} shows the results and the computation time for
different values of~$q$ for the continuous case.
As we can see in the table, the actual extreme DFFs are much fewer than the vertices of the polytope $\Phi_C(q)$. PPL is faster when $q$ is small and Normaliz performs well when $q$ is relatively large. We can observe that the time cost increases dramatically as $q$ gets large. Similar to \cite{koeppe-zhou:extreme-search}, we can apply the preprocessing
program ``redund'' provided by lrslib (version 5.08), which removes
redundant inequalities using Linear Programming. However, in contrast to the computation in \cite{koeppe-zhou:extreme-search}, removing redundancy from the system does not improve the efficiency.  Instead, for relatively large $q$, the time cost after preprocessing is a little more than that of before preprocessing for both PPL and Normaliz.

\begin{landscape}
\begin{table}[tp]
\caption{Search for extreme DFFs and efficiency of vertex enumeration codes (continuous case)} 
\centering 
\def\arraystretch{1.1}
\begin{tabular}{d{3.2} d{3.2} d{3.2} d{3.2} d{3.2} d{3.2}  d{3.2} d{3.2} } 
\toprule
\multicolumn{6}{c}{Polytope $\Phi_C(q)$} & \multicolumn{2}{c}{Running times\,(s)}\\
\cmidrule(lr){1-6}\cmidrule(lr){7-8} 
\multicolumn{1}{c}{$q$}  
    &   \multicolumn{1}{c}{dim} 
        &   \multicolumn{2}{c}{inequalities}    
    &        \multicolumn{1}{c}{vertices}  
    &   \multicolumn{1}{c}{extreme DFF} 
        &   \multicolumn{1}{c}{PPL}
            &   \multicolumn{1}{c}{Normaliz }     \\
    \cmidrule(lr){3-4}
   & & \multicolumn{1}{c}{original} & \multicolumn{1}{c}{minimized} 
                \\ [0.5ex] 
\midrule 
2 & 0 & 4&3 & 1 &1& 0.00006& 0.002\\ 
3 & 1 &  5&5& 2 &1 &0.00009& 0.006\\
5 & 2 &  9&7& 3 & 2&0.00014& 0.007\\
7 & 3 &  15&10& 5 &3 &0.0002& 0.007\\
9 & 4 &  23&14& 9 & 3&0.0004& 0.008\\ 
11 & 5 &  33&18& 14 & 7&0.0006& 0.010\\
13 & 6 & 45 &23& 25 & 8&0.001 & 0.012\\
15 & 7 & 59&29 & 66 & 14&0.003& 0.018\\
17 & 8 &  75&35& 94 & 22&0.005& 0.025\\
19 & 9 & 93&42 & 221 & 32&0.010 & 0.042\\
21 & 10 & 113&50 & 677  &55 &0.036 & 0.105\\
23 & 11 &  135&58& 1360 & 105&0.110 & 0.226\\
25 & 12 &  159&67& 3898 & 189&0.526 & 0.725\\
27 & 13 &  185&77& 12279 & 291&5.1 & 2.991\\
29 & 14 & 213&87 & 28877 & 626&41 & 9.285\\
31 & 15 &  243&98 & 91761 & 1208&595 & 35.461\\
\bottomrule 
\end{tabular}
\label{table:nonlin} 

\end{table}

\end{landscape}

\subsection{Continuous 2-slope DFFs}

In regards to continuous classical extreme DFFs, we observe most of them are 2-slope functions by computer-based search. In contrast to Gomory--Johnson 2-slope Theorem \cite{infinite}, not all 2-slope maximal classical DFFs are extreme. 

We know one of the slope values is always 0 for any 2-slope extreme classical DFF.  However, this necessary condition is not sufficient for extremality. For example, $\phi_{bj,1}(C)=(\floor{Cx}+\max(0,\frac{\{Cx\}-\{C\}}{1-\{C\}}))/\floor{C}
$ is a 2-slope function and one of the slope values is 0, but it is not extreme when $1<C<2$. The reason is there is an uncovered interval $[1-\frac{1}{C}, \frac{1}{C}]$ and we can construct an effective equivariant perturbation on the interval.

Unlike Gomory--Johnson 2-slope functions, even if we assume all intervals are
covered, 2-slope DFFs still can not guarantee there are at most 2 covered
components.  We found a continuous 2-slope extreme DFF with 3 covered
components by using computer-based search (see
\autoref{fig:2s-3c}). Therefore, the technique for proving Gomory--Johnson
2-slope Theorem no longer works in the DFF setting.

\begin{figure}[!htb]
 \includegraphics[width=\linewidth]{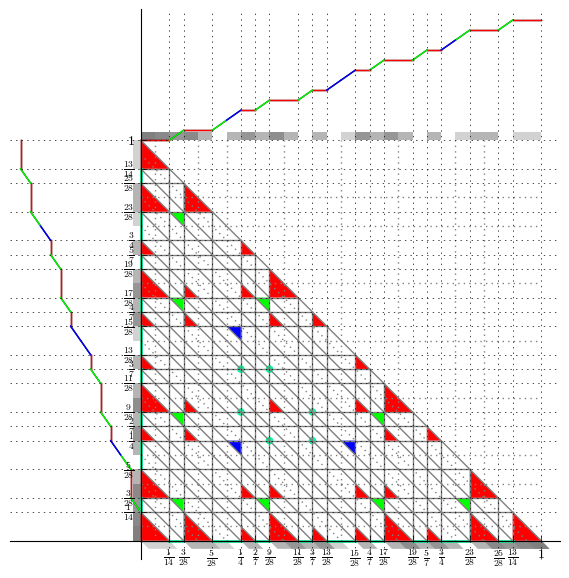}
\caption{A continuous 2-slope extreme DFF with 3 covered components for $q=28$. We use 3 different colors to color additive faces to represent 3 different covered components. The colors on the function are consistent with the colors of additive faces. We plot the function on the left and upper border. The shadows represent covered components from the projections of additive faces in 3 directions.}
\label{fig:2s-3c}
\end{figure}

\begin{conjecture}
Suppose a continuous piecewise linear maximal classical DFF has only 2 values for the derivative wherever it exists (2 slopes) and one slope value is 0. If it has no uncovered components, then it is extreme. 
\end{conjecture}

\end{document}